\documentclass[12pt, reqno]{amsart}
\usepackage{amssymb,latexsym,amsmath,amscd,amsfonts}
\usepackage{latexsym}
\usepackage[mathscr]{eucal}
\usepackage{bm}
\usepackage{mathptmx}

9.5in \numberwithin{equation}{section}

\newcommand{\beg}{\begin{equation}}
\newcommand{\eeg}{\end{equation}}
\newcommand{\ben}{\begin{eqnarray*}}
	\newcommand{\een}{\end{eqnarray*}}

\newtheorem{thm}{Theorem}[section]

\newtheorem{lem}[thm]{Lemma}

\newtheorem{prop}[thm]{Proposition}
\numberwithin{equation}{section} 
\theoremstyle{definition}
\newtheorem{defn}[thm]{Definition}
\newtheorem{rem}[thm]{Remark}

\newtheorem{eg}[thm]{Example}

\newcommand{\C}{\mathbb{C}}

\newcommand{\ov}{\overline}

\begin{document}
\title[The symmetrization map and $\Gamma$-contractions]
{The symmetrization map and $\Gamma$-contractions}

\author[Sourav Pal]{Sourav Pal}
\address[Sourav Pal]{Mathematics Department,
Indian Institute of Technology Bombay, Powai, Mumbai - 400076.}
\email{sourav@math.iitb.ac.in}



\thanks{The author is supported by the Seed Grant of IIT Bombay, the CPDA of the Govt. of India and the MATRICS Award of SERB, (Award No. MTR/2019/001010) of DST, India.}

\begin{abstract}

The symmetrization map $\pi:\mathbb C^2\rightarrow \mathbb C^2$ is defined by
$
\pi(z_1,z_2)=(z_1+z_2,z_1z_2).
$
The closed symmetrized bidisc $\Gamma$ is the symmetrization of
the closed unit bidisc $\overline{\mathbb D^2}$, that is,
\[
\Gamma = \pi(\overline{\mathbb D^2})=\{ (z_1+z_2,z_1z_2)\,:\,
|z_i|\leq 1, i=1,2 \}.
\]
A pair of commuting Hilbert space operators $(S,P)$ for which $\Gamma$ is a spectral set is called a
$\Gamma$-contraction. Unlike the scalars in $\Gamma$, a $\Gamma$-contraction may not arise as a symmetrization of a pair of commuting contractions, even not as a symmetrization of a pair of commuting bounded operators. We characterize all $\Gamma$-contractions which are symmetrization of pairs of commuting contractions. We show by constructing a family of examples that even if a $\Gamma$-contraction $(S,P)=(T_1+T_2,T_1T_2)$ for a pair of commuting bounded operators $T_1,T_2$, no real number less than $2$ can be a bound for the set $\{ \|T_1\|,\|T_2\| \}$ in general. Then we prove that every $\Gamma$-contraction $(S,P)$ is the restriction of a $\Gamma$-contraction $(\widetilde S, \widetilde P)$ to a common reducing subspace of $\widetilde S, \widetilde P$ and that $(\widetilde S, \widetilde P)=(A_1+A_2,A_1A_2)$ for a pair of commuting operators $A_1,A_2$ with $\max \{\|A_1\|, \|A_2\|\} \leq 2$. We find new characterizations for the $\Gamma$-unitaries and describe the distinguished boundary of $\Gamma$ in a different way. We also show some interplay between the fundamental operators of two $\Gamma$-contractions $(S,P)$ and $(S_1,P)$.

\end{abstract}

\maketitle

\section{Introduction}

\vspace{0.4cm}

\noindent Throughout the paper, all operators are bounded linear operators defined on complex Hilbert spaces. We denote by $\mathbb C$, the space of complex numbers and by $ \mathbb D , \mathbb T$, the unit disc and unit circle respectively in $\mathbb C$ with centre at the origin. By $\mathcal B(\mathcal H)$ we mean the algebra of bounded operators acting on a Hilbert space $\mathcal H$. For a positive operator $Q$, $\sqrt{Q}$ denotes the unique positive square root of $Q$. A contraction is an operator whose norm is not greater than $1$. If $P$ is a contraction then its defect operator $D_P$ is given by $D_P=\sqrt{I-P^*P}$. Also for $A,B \in \mathcal B(\mathcal H)$, $[A,B]=AB-BA$.

\subsection{Motivation} The symmetrization map $\pi:\mathbb C^2\rightarrow \mathbb C^2$ is defined by $\pi(z_1,z_2)=(z_1+z_2,z_1z_2)$. The \textit{symmetrized bidisc}, denoted by $\mathbb G$, is the image under $\pi$ of the bidisc $\mathbb D^2$, that is,
\[
\mathbb G =\{(z_1+z_2,z_1z_2)\,:\, |z_i|< 1, i=1,2   \}\subset \mathbb C^2.
\]
Also, the \textit{closed symmetrized bidisc}, denoted by $\Gamma$, is the following set
\[
\Gamma =\overline{\mathbb G}= \pi(\overline{\mathbb D^2})=\{(s,p)= (z_1+z_2,z_1z_2)\,:\,
|z_i|\leq 1, i=1,2 \}\;.
\]
The symmetrized bidisc has its origin in the $2\times 2$ spectral Nevanlinna-Pick interpolation. The general $n\times n$ spectral Nevanlinna-Pick interpolation problem states the following: given distinct points $w_1,\dots,w_n$ in $\mathbb D$ and $n\times n$ matrices $A_1,\dots,A_n$ in the spectral unit ball $\Omega_n$ of the space of $n\times n$ matrices $\mathcal M_n(\mathbb C)$, whether or under what conditions it is possible to find an analytic function $f:\mathbb D \rightarrow \Omega_n$ such that $f(w_i)=A_i$, $i=1,\dots,n$. It is obvious that a $2 \times 2$ matrix $A$ is in $\Omega_2$ if and only if its eigenvalues $\sigma_1,\sigma_2$ are in $\mathbb D$ and this happens if and only if $\left(\text{tr} (A),\det (A)\right)=(\sigma_1+\sigma_2,\sigma_1\sigma_2)$ belongs to $\mathbb G$. The main motivation behind studying the symmetrized bidisc is that the $2\times 2$ spectral Nevanlinna-Pick interpolation problem reduces to a similar interpolation problem of $\mathbb G$ in the following way.

\begin{prop}[\cite{N-P-T}, Proposition 1.1]\label{Nikolov}

Let $\alpha_1,\dots, \alpha_n\in\mathbb D$ be distinct points and let $A_1=\lambda_1 I,\dots, A_k=\lambda_kI\in \Omega_2$ be scalar matrices. Also let $A_{k+1},\dots, A_n \in \Omega_2$ be non-scalar matrices. Suppose $\phi=(\phi_1,\phi_2):\mathbb D \rightarrow \mathbb G$ is a holomorphic map such that $\phi(\alpha_j)=\sigma(A_j)=\left(\text{tr} (A_j),\det (A_j)\right)$ for $j=1,\dots ,n$. Then there exists a holomorphic map $\psi:\mathbb D \rightarrow \Omega_2$ satisfying $\phi=\sigma \circ \psi$ and $\psi (\alpha_j)=A_j$ for $j=1,\dots , n$ if and only if $\phi_2'(\alpha_j)=\lambda_j \phi_1'(\alpha_j)$ for $j=1,\dots , k$.
\end{prop}

Obviously a bounded domain like $\mathbb G$, which has complex-dimension $2$, is much easier to deal with than a norm-unbounded object like $\Omega_2$ which has complex-dimension $4$. On the other hand, $\mathbb G$ was the first example of a non-convex domain in which the Caratheodory and Kobayashi distances coincide (see \cite{ay-jga}). Apart from such beautiful complex analytic aspects (see \cite{Pf-Zwo, B-S1, Ag-Ly-Yo1, Ag-Ly-Yo2} etc. for more complex analytic and geometric results), operator theory on the symmetrized bidisc was initiated in \cite{ay-jfa} to make a new approach to the $2 \times 2$ spectral Nevanlinna-Pick problem. Later, operator theory on $\Gamma$ turned out to be of independent interest and has been extensively studied in past two decades by several mathematicians, e.g. \cite{ay-pems, ay-jot, T-H, tirtha-sourav, tirtha-sourav1, sourav, pal-shalit, J-S, B-S2, B-D-S} and many more (see the references therein). In this article, we study is a commuting pair of operators for which $\Gamma$ is a spectral set.

\begin{defn}
A pair of commuting operators $(S,P)$ is called a $\Gamma$-\textit{contraction} if $\Gamma$
is a spectral set for $(S,P)$, that is, the Taylor joint spectrum $\sigma_T
(S,P)\subseteq \Gamma$ and the von-Neumann's inequality
\begin{equation}\label{eqn:I01}
\|f(S,P) \|\leq \sup_{(z_1,z_2)\in\Gamma}|f(z_1,z_2)| =\|f\|_{\infty,\, \Gamma}
\end{equation}
holds for all rational functions $f=p/q$ with $p, q \in \C[z_1,z_2]$ and $q$ does not have any zero inside $\Gamma$. 
\end{defn}
Note that, it follows from the definition that if $(S,P)$ is a $\Gamma$-contraction then $\|S\|\leq 2$ and $\|P\|\leq 1$. By virtue of polynomial convexity of $\Gamma$ (see Lemma 2.1 in \cite{ay-jfa}), it suffices if the von-Neumann's inequality (\ref{eqn:I01}) holds only for the polynomials in $\mathbb C[z_1,z_2]$  and an elementary proof is given below. Also, $(S,P)$ is called a \textit{complete} $\Gamma$-\textit{contraction} if $\Gamma$ is a complete spectral set for $(S,P)$, that is (\ref{eqn:I01}) holds when $f$ is any matricial polynomial.

\subsection{A brief outline of the results}

It follows from Ando's inequality that $(T_1+T_2, T_1T_2)$ is a $\Gamma$-contraction when $T_1,T_2$ are commuting contractions. Example 1.7 in \cite{ay-jot} shows that the converse is not true, i.e. not every $\Gamma$-contraction $(S,P)$ arrises as symmetrization of a pair of commuting contractions. Indeed, $(S,P)$ mat not be even symmetrization of a pair of commuting operators. Also, a remark was made in \cite{ay-jot} in this context without a proof about a characterization of the $\Gamma$-contractions which are symmerization of commuting contractions. In Section \ref{sec:02}, we show by an example that this remark is partially correct and in Theorem \ref{lem:char} we characterize all such $\Gamma$-contractions $(S,P)$ which are symmetrization of commuting contractions. We show by an explicit example that the condition $\|S\pm \Delta \| \leq 2$ cannot be ignored in this case, where $\Delta$ is a square root of $S^2-4P$ that commutes with $S,P$. We also show by examples that if $(S,P)=(T_1+T_2,T_1T_2)$, then $T_1,T_2$ may not be unique unlike the scalars in $\Gamma$. A natural question arises: if $(S,P)=(T_1+T_2,T_1T_2)$, then can we put a bound on the norm of $T_1, T_2$ ? We construct a family of $\Gamma$-contractions $\mathcal F=\{ (S_r,0):r\in \Lambda \}$ such that each member $(S_r, 0)$ is symmetrization of a unique pair, namely $\{S_r,0\}$ itself and that for every $\delta >0$, there is $(S_{\widetilde{r}},0)\in \mathcal F$ satisfying $2-\delta < \|S_{\widetilde{r}} \| \leq 2$. This shows that no real number less than $2$ can be a bound for $\{ \|T_1\|,\|T_2\| \}$ in general. Hence, we propose the following theorem which is the main result of this paper.

\begin{thm}\label{thm:N03}
For any $\Gamma$-contraction $(S,P)$, acting on a Hilbert space $\mathcal H$, there exists a Hilbert space $\mathcal K$ containing $\mathcal H$ as a closed linear subspace and a pair of commuting operators $T_1,T_2$ on $\mathcal K$ such that
\begin{itemize}
\item[(i)] $(T_1+T_2,T_1T_2)$ is a $\Gamma$-contraction ;
\item[(ii)] $\mathcal H$ is a joint reducing subspace for $T_1+T_2$ , $T_1T_2$ ;
\item[(iii)] $(T_1+T_2)|_{\mathcal H}=S$ , $(T_1T_2)|_{\mathcal H}=P$ ;
\item[(iv)] $\max \{ \|T_1\|,\,\|T_2\| \} \leq 2$.
\end{itemize}
\end{thm}

We turn our focus to some special classes of $\Gamma$-contractions, namely the $\Gamma$-{unitaries} and the $\Gamma$-{isometries}.
\begin{defn}
Let $S,P$ be commuting operators acting on a Hilbert space $\mathcal H$. Then
\begin{itemize}
\item[(i)] $(S,P)$ is a $\Gamma$-\textit{unitary} if $S,P$ are normal operators and the Taylor joint spectrum $\sigma_T(S,P)$ is a subset of the distinguished boundary $b\Gamma$ ;

\item[(ii)] $(S,P)$ is a $\Gamma$-\textit{isometry} if there is a Hilbert space $\mathcal K$ that contains $\mathcal H$ as a closed linear subspace and a $\Gamma$-unitary $(\widetilde S, \widetilde P)$ on $\mathcal K$ such that $\mathcal H$ is a joint invariant subspace of $\widetilde S, \widetilde P$ and that $\widetilde S|_{\mathcal H}=S$, $\widetilde P|_{\mathcal H}=P$.
\end{itemize}
\end{defn}
In Section \ref{sec:05}, we present a new characterization for the $\Gamma$-unitaries and also characterize the points in $\Gamma \setminus b\Gamma$.\\

In \cite{ay-jfa}, Agler and Young profoundly established the fact that $\Gamma$ is a spectral set for $(S,P)$ if and only if it is a complete spectral set for $(S,P)$ (see Theorem 1.2 in \cite{ay-jfa} ). The same was proved in \cite{tirtha-sourav} by constructing an explicit $\Gamma$-unitary dilation of a $\Gamma$-contraction (\cite{tirtha-sourav}, Theorem 4.3). A new and effective machinery namely \textit{fundamental operator}, was introduced in \cite{tirtha-sourav} to construct that explicit dilation. Indeed, in Theorem 4.2 of \cite{tirtha-sourav}, it was shown that for every $\Gamma$-contraction $(S,P)$, there exists a unique operator $F$ in $\mathcal B(\mathcal D_P)$ (where $\mathcal D_P=\overline{Ran}\,D_P$) such that the numerical radius of $F$ is not greater than $1$ and that 
\begin{equation}\label{eqn:I02}
S-S^*P=D_PFD_P.
\end{equation}
This unique operator $F$ is called the {fundamental operator} of the $\Gamma$-contraction $(S,P)$. In this context we ask the following question.\\

\noindent \textit{\textbf{Question 2.}} Let $(S,P)$ and $(S_1,P)$ be two $\Gamma$-contractions on a Hilbert space $\mathcal H$ such that $S,S_1$ commute. Suppose $(T,U)$ and $(T_1,U)$ on $\mathcal K\;(\supseteq \mathcal H)$ are $\Gamma$-unitary dilations of $(S,P)$ and $(S_1,P)$ respectively, where $\mathcal K$ is the minimal unitary dilation space for $P$. Then do $T,T_1$ necessarily commute ?\\

\noindent It is merely said that if $T,T_1$ are any lifts of $S,S_1$ respectively, then they do not necessarily commute. But, here in this context we have considered a special kind of commutant lifting and the lifts are defined on the minimal dilation space of $P$. Also, they are the first components of two $\Gamma$-unitaries having the same last component. We answer the above question in Section \ref{sec:06}. In Section 2, we accumulate a few results from the literature which will be used in sequel.

\vspace{0.5cm}

\section{Preliminaries}

\vspace{0.5cm}

\noindent In this Section, we recollect from literature a few results which we shall use in sequel. First, we state a theorem from \cite{ay-jot} that characterizes a point in $\Gamma$ in several ways.

\begin{thm}[\cite{ay-jot}, Theorem 1.1]\label{thm:char-gamma}
Let $(s,p)\in\mathbb C^2$. Then the following are equivalent:
\begin{itemize}
 \item[(i)] $(s,p)\in \Gamma$ ;
 
 \item[(ii)] $|s-\overline{s}p|+|p|^2 \leq 1$ and $|s|\leq 2$ ;
 
 \item[(iii)] $2|s-\overline{s}p|+|s^2-4p|+|s|^2 \leq 4$ ;
 
 \item[(iv)] $|p|\leq 1$ and there exists $\beta \in\mathbb C$ such that $|\beta|\leq 1$ and $s= \overline{\beta} + \beta p$.

\end{itemize}
\end{thm}

The following result provides a characterization for the points in the distinguished boundary $b\Gamma$ of $\Gamma$.

\begin{thm}[\cite{tirtha-sourav}, Theorem 2.5]\label{thm:char-DB}
Let $(s,p)$ be a point in $\mathbb C^2$. Then $(s,p)\in b \Gamma$ if and only if $(s,p)\in \Gamma$ and $|p|=1$.
\end{thm}

The rational dilation succeeds on the symmetrized bidisc. As a consequence of this, we obtain from literature the following appealing characterizations of a $\Gamma$-contraction.

\begin{thm}[\cite{ay-jfa}, Theorem 1.2 \& \cite{tirtha-sourav}, Theorem 4.3]\label{char}

Let $(S,P)$ be a pair of commuting operators on a Hilbert space
$\mathcal H$. Then the following are equivalent.
\begin{enumerate}
\item $(S,P)$ is a $\Gamma$-contraction ; \item $(S,P)$ is a
complete $\Gamma$-contraction ; \item $\sigma(S,P)\subseteq
\Gamma$ and $\rho(\alpha S,\,{\alpha}^2P)\geq 0,$ for all
$\alpha\in\mathbb D$, where
\[
\rho(S,P)=2(I-P^*P)-(S-S^*P)-(S^*-P^*S)\; ;
\]

\item $\|S\|\leq 2\,,\, \|P\|\leq 1$ and the operator equation
$S-S^*P=D_PXD_P$ has a unique solution $F$ in $\mathcal B(\mathcal
D_P)$ with $\omega(F)\leq 1$, where $D_P=(I-P^*P)^{\frac{1}{2}}$
and $\mathcal B(\mathcal D_P)$ is the algebra of bounded operators
on $\mathcal D_P=\overline{Ran} D_P$.
\end{enumerate}

\end{thm}

Apart from the geometric definition of a $\Gamma$-unitary, the following theorem from \cite{ay-jot} describes a $\Gamma$-unitary algebraically in different ways.

\begin{thm}[\cite{ay-jot}, Theorem 2.2]\label{thm:char-gamma-unitary}
Let $S,P$ be commuting operators on a Hilbert space $\mathcal H$. Then the following are equivalent:
\begin{itemize}
\item[(i)] $(S,P)$ is a $\Gamma$-unitary ;

\item[(ii)] $P^*P=I=PP^*$ and $P^*S=S^*$ and $\|S\|\leq 2$ ;

\item[(iii)] there exist commuting unitary operators $U_1,U_2$ on $\mathcal H$ such that $S=U_1+U_2$ and $P=U_1U_2$.
\end{itemize}
\end{thm}

The following is a structure theorem for a $\Gamma$-isometry which also gives a few characterizations of it.

\begin{thm}[\cite{ay-jot}, Theorem 2.6] \label{thm:char-gamma-isometry}
Let $S,P$ be commuting operators on a Hilbert space $\mathcal H$. Then the following statements are equivalent:
\begin{itemize}

\item[(i)] $(S,P)$ is a $\Gamma$-isometry ;

\item[(ii)] there is an orthogonal decomposition $\mathcal H =\mathcal H_1 \oplus \mathcal H_2$ into common reducing subspaces of $S$ and $P$ such that $(S|_{\mathcal H_1}, P|_{\mathcal H_1})$ is a $\Gamma$-unitary and $(S|_{\mathcal H_2},P|_{\mathcal H_2})$ is a pure $\Gamma$-isometry ;

\item[(iii)] $P^*P=I$ and $P^*S=S^*$ and $\|S\|\leq 2$.

\end{itemize}

\end{thm}

\vspace{0.18cm}

\section{$\Gamma$-contractions as symmetrization of commuting contractions}\label{sec:02}

\vspace{0.5cm}

\noindent Recall that the
points in $\Gamma$ are nothing but the symmetrization of the points in $\overline{\mathbb D^2}$. It was shown in
\cite{ay-jfa} (see \cite{ay-jfa}, Section 2) by an application
of Schur's theorem that $\pi^{-1}(\Gamma)=\overline{\mathbb D^2}$.
In fact, corresponding to every point $(s,p)$ in $\C^2$,
there are two points $z_1=\frac{1}{2}(s+\sqrt{s^2-4p})$ and
$z_2=\frac{1}{2}(s-\sqrt{s^2-4p}) \}$
such that $\pi^{-1}(s,p)=\{ (z_1,z_2),(z_2,z_1) \}$. So, when $(s,p)\in\Gamma$, the points $(z_1,z_2)$ and
$(z_2,z_1)$ are in $\overline{\mathbb D^2}$. Unlike the scalars,
the picture is different when we consider operator pairs $(S,P)$
that have $\Gamma$ as a spectral set. Not every
$\Gamma$-contraction is a symmetrization of two commuting
operators. We mention here an example given by Agler and Young (see \cite{ay-jot}, Example 1.7). If $P$ is a contraction then by
part-(3) of Theorem \ref{char} $(0,P)$ is a $\Gamma$-contraction.
Now if $-P$ does not have any square root then $(0,P)$ can not be expressed as
$(T_1+T_2,T_1T_2)$ for two operators $T_1,T_2$.
Since all $\Gamma$-contractions do not arise as symmetrization of
a pair of commuting contractions, it is worth finding out a
characterization for the $\Gamma$-contractions which are
symmetrization of two commuting contractions. A possible direction was proposed by Agler and Young in \cite{ay-jot} without a proof which we state here: a $\Gamma$-contraction $(S,P)$ is the symmetrization of a pair of commuting contractions $T_1,T_2$ if and only if $S^2-4P$
has a square root say $\Delta$ that commutes with $S$ and $P$. This proposition leads us to the right conclusion at the cost of a norm bound for $S\pm \Delta$ as we see in Theorem \ref{lem:char} below. Also, the norm bound on $\|S \pm \Delta \|$ cannot be ignored (see Example \ref{exm:t01}). If $(S,P)$ is a $\Gamma$-contraction such that
$S^2-4P$ possesses a square root say $\sqrt{S^2-4P}$ that commutes with $S,P$
and if $S=T_1+T_2\, , \, P=T_1T_2$ then a canonical choice for
$T_1,T_2$ is that

\[
T_1 = \dfrac{1}{2}\{ S+\sqrt{S^2-4P} \}\;,\; T_2 = \dfrac{1}{2}\{
S-\sqrt{S^2-4P} \} \,.
\]
But, we can not conclude anything about the norms of such
$T_1, T_2$. This is because unlike the scalar case, the fact that
the symmetrization of a pair of commuting operators $T_1,T_2$ is a
$\Gamma$-contraction does not imply that $T_1$ and $T_2$ are
contractions. For example, let us consider the commuting matrices
\begin{equation} \label{eqn:fu01}
T_1 =\begin{bmatrix}
0 & z \\
0 & 0
\end{bmatrix}\;,\;
T_2 = \begin{bmatrix}
0 & -z \\
0 & 0
\end{bmatrix} \qquad (z\in \C),
\end{equation}
which are not contractions but $(T_1+T_2,T_1T_2)=(0,0)$
which is a $\Gamma$-contraction by part-(3) of Theorem \ref{char}. So, let us characterize all $\Gamma$-contractions that are symmetrization of pairs of commuting contractions.

\begin{thm}\label{lem:char}
Let $(S,P)$ be a $\Gamma$-contraction acting on a Hilbert space
$\mathcal H$. Then $S=T_1+T_2$ and $P=T_1T_2$ for a pair of
commuting contractions $T_1,T_2$ on $\mathcal H$ if and only if
$S^2-4P$ has a square root $\Delta$ such that

\begin{itemize}
\item[(i)] $\Delta$ commutes with $S$ and $P$ \item[(ii)] $\|S \pm
\Delta\| \leq 2$.
\end{itemize}

\end{thm}

\begin{proof}
Let there be two commuting contractions $T_1,T_2$ such that
$S=T_1+T_2$ and $P=T_1T_2$. Then $S^2-4P=(T_1-T_2)^2$ and clearly
$S^2-4P$ has a square root namely $\Delta=T_1-T_2$ which commutes
with $S,P$ and the norm conditions $\|S \pm \Delta\| \leq 2$ are
satisfied.

Conversely, suppose that $S^2-4P$ has a square root $\Delta$ that
commutes with $S,P$ and satisfies $\|S \pm \Delta\| \leq 2$.
Setting $T_1=\frac{1}{2}(S+\Delta)$ and
$T_2=\frac{1}{2}(S-\Delta)$ we see that $T_1, T_2$ commute as
$\Delta$ commutes with $S$ and $P$. The fact that $T_1, T_2$ are
contractions follows from the inequalities $\|S \pm \Delta\| \leq
2$. Also $T_1+T_2=S$ and $T_1T_2=P$.
\end{proof}

\noindent \textbf{\textit{Non-uniqueness}.} Let $\mathcal F$ be the class of $\Gamma$-contractions that are symmetrization of pairs of commuting contractions. Unlike the scalars $(s,p)\in\Gamma$, a $\Gamma$-contraction $(S,P)\in\mathcal F$ may not arise as the symmetrization of a unique pair of commuting contractions. In fact there could be infinitely many such pairs as we witnessed in (\ref{eqn:fu01}).\\

The condition-(ii), i.e., $\| S\pm \Delta \|\leq 2$ in Theorem \ref{lem:char} cannot be ignored. The following example will verify this. Before going to the example we state an useful result from \cite{ay-jot} which will be used below.
\begin{lem} [\cite{ay-jot}, Corollary 1.9] \label{lem:AY01}
Let $(S,P)$ be a commuting pair of operators such that $\|P\| < 1$ and the spectral radius of $S$ is less than $2$. Then $(S,P)$ is a $\Gamma$-contraction if and only if 
\[
\omega \left( (I-P^*P)^{-\frac{1}{2}} (S-S^*P) (I-P^*P)^{- \frac{1}{2}} \right) \leq 1.
\]

\end{lem}

\begin{eg} \label{exm:t01}
 Let us consider the following $2 \times 2$ commuting scalar matrices
\[
S_{\epsilon} = \begin{bmatrix}
{\epsilon} & {\epsilon} \\
0 & 0
\end{bmatrix}
, \quad {\epsilon}>0 \quad \text{ and } \quad 
P = \begin{bmatrix}
0 & 0 \\
0 & 0
\end{bmatrix}.
\]
Clearly $\| S_{\epsilon} \| \geq \sqrt{2} {\epsilon} $ (Considering the action of $S_{\epsilon}$ on the transpose of the vector $[\frac{1}{\sqrt 2}, \frac{1}{\sqrt 2}]$.) Now we calculate the numerical radius of $S_{\epsilon}$. Let $(\alpha , \beta) \in \mathbb C^2$ be arbitrary with $\|(\alpha, \beta) \|=1$, i.e. $|\alpha|^2+|\beta|^2=1$. Now
\[
\left| \left\langle 
\begin{bmatrix}
{\epsilon} & {\epsilon} \\
0 & 0
\end{bmatrix} 
\begin{bmatrix}
\alpha \\
\beta
\end{bmatrix} \,,\;
\begin{bmatrix}
\alpha \\
\beta
\end{bmatrix} 
\right\rangle
\right|
=\left| {\epsilon}(|\alpha|^2+ \beta \bar{\alpha}) \right|
\leq {\epsilon} |\alpha|(|\alpha|+|\beta|).
\]
Finding the maximum value of $|\alpha|(|\alpha|+|\beta|)$, where $|\alpha|^2+|\beta|^2=1$, is same as
finding the maximum value of $f(\theta)=\sin \theta (\sin \theta + \cos \theta)$. Evidently
\begin{align*}
\sin \theta (\sin \theta + \cos \theta) & = \sqrt 2 \sin \theta \sin(\theta + \frac{\pi}{4}) \\
& = \dfrac{\sqrt 2}{2} \left\{\cos \frac{\pi}{4} - \cos (2\theta + \frac{\pi}{4})  \right\} \\
& = \dfrac{\sqrt 2}{2} \left\{\dfrac{1}{\sqrt 2} - \cos (2\theta + \frac{\pi}{4})  \right\},
\end{align*}
which is maximum when $\cos (2\theta + \frac{\pi}{4})=-1$. Hence the maximum value of $\sin \theta (\sin \theta + \cos \theta)$ is equal to $\dfrac{\sqrt 2}{2}\left( \dfrac{1}{\sqrt 2} + 1 \right)=\dfrac{\sqrt 2 +1}{2} \approx 1.207 < 1.22$. Therefore, $\omega (S_{\epsilon})< 1.22 {\epsilon}$ for any ${\epsilon} >0$. Choosing $\widetilde {\epsilon}=\dfrac{1}{1.3}$ we see that $\omega (S_{\widetilde{\epsilon}})< 1$ and that $\| S_{\widetilde{\epsilon}}\| \geq \dfrac{\sqrt 2}{1.3} >1$. Thus, $(S_{\widetilde{\epsilon}}, P)$ is a pair of commuting operators such that $\|P\|<1$, $r(S_{\widetilde{\epsilon}})<2$ and $\omega(S_{\widetilde{\epsilon}})<1$. So, by Lemma \ref{lem:AY01}, $(S_{\widetilde{\epsilon}},P)$ is a $\Gamma$-contraction.\\

Note that for any $\epsilon >0$, $S_{\epsilon}^2-4P$ has a square root namely $S_{\epsilon}$ which commutes with both $S_{\epsilon}$ and $P$. We now show that the $\Gamma$-contraction $(S_{\widetilde{\epsilon}} , P)$ cannot be the symmetrization of a pair of commuting contractions. Clearly, $(S_{\widetilde{\epsilon}}, 0)=\pi (S_{\widetilde{\epsilon}}, 0)$. Since $\| S_{\widetilde{\epsilon}}\|>1$, it suffices to show that $(S_{\widetilde{\epsilon}}, 0)$ is the unique pair such that $(S_{\widetilde{\epsilon}}, 0)=\pi (S_{\widetilde{\epsilon}}, 0)$. Let $T_1,T_2 \in \mathcal M_2(\C)$ be such that $T_1+T_2=S_{\epsilon}$ and $T_1T_2=0$. So, without loss of generality let us assume that
\[
T_1= \begin{bmatrix}
a_1 & a_2 \\
a_3 & a_4
\end{bmatrix}
\quad \text{ and } \quad T_2= \begin{bmatrix}
\epsilon - a_1 & \epsilon - a_2 \\
-a_3 & -a_4
\end{bmatrix}
\]
and that $T_1T_2=T_2T_1=0$. Now
\[
T_1T_2=\begin{bmatrix}
a_1({\epsilon}-a_1)-a_2a_3 & a_1({\epsilon}-a_2)-a_2a_4 \\
a_3({\epsilon}-a_1)-a_3a_4  & a_3({\epsilon} - a_2) - a_4^2
\end{bmatrix}
\]
and
\[
T_2T_1
= \begin{bmatrix}
a_1({\epsilon}-a_1)+a_3({\epsilon}-a_2) & a_2({\epsilon}-a_1)+a_4({\epsilon}-a_2) \\
-a_1a_3-a_3a_4 & -a_2a_3 - a_4^2
\end{bmatrix}.
\]
Since $T_1T_2=T_2T_1=0$, considering the $(1,1)$ and $(1,2)$ blocks, we obtain
\begin{align}
& a_1({\epsilon} - a_1)-a_2a_3=a_1({\epsilon} - a_1)+ a_3({\epsilon} -a_2)=0 \label{eqn:t01}\\
& a_1({\epsilon}-a_2)-a_2a_4=a_2({\epsilon} - a_1) + a_4({\epsilon} - a_2)=0. \label{eqn:t02}
\end{align}
From the first equality of (\ref{eqn:t01}) we have that $a_3=0$ as ${\epsilon} >0$ and also, from the first equality of (\ref{eqn:t02}) we have $a_1=a_2+a_4$. Thus
\[
T_1= \begin{bmatrix}
a_1 & a_2 \\
0 & a_1-a_2
\end{bmatrix}\,,\,
T_2= \begin{bmatrix}
{\epsilon} - a_1 & {\epsilon} - a_2 \\
0 & a_2-a_1
\end{bmatrix}
\]
and hence
\[
T_1T_2= \begin{bmatrix}
a_1({\epsilon} -a_1) & a_1({\epsilon} - a_2) + a_2(a_2-a_1) \\
0 & -(a_1-a_2)^2
\end{bmatrix}
=
\begin{bmatrix}
0 & 0 \\
0 & 0
\end{bmatrix}.
\]
Considering the $(2,2)$ block we obtain $a_1=a_2$. Thus
\[
T_1T_2=
\begin{bmatrix}
a_1({\epsilon} - a_1) & a_1({\epsilon} -a_1) \\
0 & 0
\end{bmatrix}.
\]
Now $a_1({\epsilon}-a_1)=0$ implies that either ${\epsilon}=a_1$ or $a_1=0$. In either cases we have
\[
\{ T_1 , T_2 \}= \left\{ \begin{bmatrix}
{\epsilon} & {\epsilon} \\
0 & 0
\end{bmatrix}\,,\,
\begin{bmatrix}
0 & 0 \\
0 & 0
\end{bmatrix}
\right\} = \{ S_{\epsilon} , P \},
\]
and consequently our claim is established.
 
\end{eg}

\begin{rem}

Thus, we saw that a $\Gamma$-contraction $(S,P)$ can be symmetrization of a unique pair of commuting operators $T_1,T_2$ but $T_1,T_2$ may not be contractions as in the example above $\max \{ \|S_{\widetilde{\epsilon}}\|, \|P\| \}>1$. This triggers a natural question whether we can make an estimate for a bound on the norms of $T_1,T_2$ when $(S,P)=(T_1+T_2,T_1T_2)$. The following example shows that in general no bound less than $2$ works in such case.

\end{rem}

\begin{eg} \label{exm:02}

Let us consider the following family of commuting pairs of $2 \times 2$ scalar matrices:
\[
\mathcal F = \left\{ (S_r,0)\,:\; S_r =
\begin{bmatrix}
\frac{r^2}{2} & 2-r \\
0 & \frac{r^2}{2}
\end{bmatrix} \text{ with } 0<r< \dfrac{1}{100} \right \}.
\]
We show that each $(S_r,0)$ is a $\Gamma$-contraction which is the symmetrization of a unique pair of commuting operators $S_r ,0$ and that for every $\delta >0$ (with $2-\delta >0$), there is a member of the family $\mathcal F$, say $(S_{\widehat r}, 0)$ such that $2- \delta < \|S_{\widehat r} \| \leq 2$. Indeed, this will suffice to establish that no bound less than $2$ works in general.\\

We first show that $\omega(S_r)<1$ for each $(S_r,0)\in \mathcal F$. For any unit vector $(\alpha ,\beta)\in \C^2$, we have
\begin{align*}
\left\langle 
\begin{bmatrix}
\frac{r^2}{2} & {2-r} \\
0 & \frac{r^2}{2}
\end{bmatrix} 
\begin{bmatrix}
\alpha \\
\beta
\end{bmatrix} \,,\;
\begin{bmatrix}
\alpha \\
\beta
\end{bmatrix} 
\right\rangle
 & = \left\langle 
\begin{bmatrix}
(r^2/2)\alpha + (2-r)\beta \\
(r^2/2)\beta
\end{bmatrix} \,,\;
\begin{bmatrix}
\alpha \\
\beta
\end{bmatrix} 
\right\rangle
\\
&= \dfrac{r^2}{2}+(2-r)\beta \ov{\alpha}. \qquad [\text{ since } |\alpha|^2+|\beta|^2=1]
\end{align*}
Therefore,
\[
\left| \left\langle 
\begin{bmatrix}
\frac{r^2}{2} & {2-r} \\
0 & \frac{r^2}{2}
\end{bmatrix} 
\begin{bmatrix}
\alpha \\
\beta
\end{bmatrix} \,,\;
\begin{bmatrix}
\alpha \\
\beta
\end{bmatrix} 
\right\rangle
\right|
\leq \dfrac{r^2}{2}+(2-r)|\alpha \beta|.
\]
Since $|\alpha|^2+|\beta|^2=1=\sin^2 \theta + \cos^2 \theta$, it suffices to find an upper bound say $M$ of $|\alpha \beta|=|\sin \theta \cos \theta|$ to conclude that $\omega(S_r, 0)\leq \dfrac{r^2}{2}+(2-r)M$. Note that $|\sin \theta \cos \theta|=\left| \dfrac{1}{2} \sin 2\theta \right|\leq \dfrac{1}{2}$. Therefore, $M=1$ works and we have
\[
\omega(S_r)\leq \dfrac{r^2}{2}+ \dfrac{2-r}{2}=\dfrac{2+r^2-r}{2} <1 \,, \qquad \text{ as }\;\; 0<r< \dfrac{1}{100}.
\]
So, $\|S_r\|< 2$ and it follows from Lemma \ref{lem:AY01} that each $(S_r,0)\in \mathcal F$ is a $\Gamma$-contraction. Also, it is evident that $\|S_r\| \rightarrow 2$ as $r \rightarrow 0$. Therefore, for every $\delta >0$ with $2- \delta >0$, there is a member of the family $\mathcal F$, say $S_{\widehat r}$ such that
\[
2- \delta < \|S_{\widehat r} \| \leq 2.
\]
It remains to show that $(S_r,0)$ arises as the symmetrization of a unique pair of commuting operators, which are $S_r$ and $0$. It is obvious that $(S_r,0)$ is the symmetrization of $S_r$ and $ 0 $, we need to prove the uniqueness part only. For the sake of calculation we denote $a_1=\dfrac{r^2}{2}$ and $a_2=2-r$ so that $S_r=\begin{bmatrix}
a_1 & a_2 \\
0 & a_1
\end{bmatrix}.
$
Suppose $(S_r,0)=(T_1+T_2,T_1T_2)$ for a pair of commuting $2\times 2$ matrices $T_1, T_2$ and suppose
\[
T_1=
\begin{bmatrix}
b_1 & b_2 \\
b_3 & b_4
\end{bmatrix}
\,,\; \quad T_2
=
\begin{bmatrix}
a_1-b_1 & a_2 - b_2 \\
-b_3 & a_1-b_4
\end{bmatrix}.
\]
Then
\[
T_1T_2=
\begin{bmatrix}
b_1(a_1-b_1)-b_2b_3 & b_1(a_2-b_2)+b_2(a_1-b_4) \\
b_3(a_1-b_1)-b_3b_4 & b_3(a_2-b_2)+ b_4(a_1-b_4) 
\end{bmatrix}
\]
and
\[
T_2T_1=
\begin{bmatrix}
b_1(a_1-b_1)+b_3(a_2-b_2) & b_2(a_1-b_1)+b_4(a_2-b_2) \\
b_3(a_1-b_4)-b_1b_3 & b_4(a_1-b_4) - b_2b_3 
\end{bmatrix}.
\]
Since $T_1T_2=T_2T_1=0$, we have the following four identities:
\begin{align}
& b_1(a_1-b_1)-b_2b_3  = b_1(a_1-b_1)+b_3(a_2-b_2) =0 \; ; \label{eqn:u01} \\
& b_1(a_2-b_2)+b_2(a_1-b_4) = b_2(a_1-b_1)+b_4(a_2-b_2) =0 \; ; \label{eqn:u02}  \\
& b_3(a_1-b_1)-b_3b_4 = b_3(a_1-b_4)-b_1b_3 =0 \; ; \label{eqn:u03} \\
& b_3(a_2-b_2)+ b_4(a_1-b_4) = b_4(a_1-b_4) - b_2b_3 =0 \label{eqn:u04}.
\end{align}
From the first equality in (\ref{eqn:u01}), we have $b_3a_2=0$ and since $a_2=2-r >0$, it follows that $b_3=0$. Again, from the first equality in (\ref{eqn:u02}), we have $a_2(b_1-b_4)=0$ which implies that $b_1=b_4$ as $a_2>0$. Thus,
\[
T_1=
\begin{bmatrix}
b_1 & b_2 \\
0 & b_1
\end{bmatrix}
\,,\; \quad T_2
=
\begin{bmatrix}
a_1-b_1 & a_2 - b_2 \\
0 & a_1-b_1
\end{bmatrix}\,,
\]
and consequently the above equations (\ref{eqn:u01}) and (\ref{eqn:u02}) reduce to the following two equations respectively:
\begin{align}
&  b_1(a_1-b_1) =0 \; ; \label{eqn:v01} \\
&  b_2(a_1-b_1)+b_1(a_2-b_2) =0 \;  \label{eqn:v02}.
\end{align}
Evidently (\ref{eqn:v01}) implies that either $b_1=0$ or $a_1=b_1$. If $b_1=0$, it follows from (\ref{eqn:v02}) that $b_2a_1=0$. Therefore, $b_2=0$ as $a_1=\dfrac{r^2}{2}>0$. Thus, $b_1=b_2=b_3=b_4=0$ and consequently $T_1=0$ and $T_2=S_r$. On the other hand if $a_1=b_1$, then from (\ref{eqn:v02}) we have $a_1(a_2-b_2)=0$. So, we have $a_2=b_2$ as $a_1>0$. It clearly shows that $T_2=0$ (as $a_1=b_1$ and $a_2=b_2$) and thus we have $T_1=S_r$. So, considering either cases we have that $\{ T_1,T_2 \}= \{ S_r, 0 \}$ and thus $\{ S_r, 0 \}$ is unique.

\end{eg}

%\vspace{0.1cm}

Thus, until now we have observed the following two important points:
\begin{itemize}
\item[(a)] Example \ref{exm:02} guarantees that in general no real number less than $2$ can be a bound for $\max \{ \|T_1\|,\|T_2\|\}$ when $(S,P)=(T_1+T_2,T_1T_2)$ for a unique pair of commuting operators $T_1,T_2$.

\item[(b)] Example 1.7 in \cite{ay-jot} shows that there are $\Gamma$-contractions which are not even symmetrization of commuting operators.
\end{itemize}
Taking cue from these two points, we propose the following theorem which is the main result of this article.

\begin{thm}
For any $\Gamma$-contraction $(S,P)$, acting on a Hilbert space $\mathcal H$, there exists a Hilbert space $\mathcal K$ containing $\mathcal H$ as a closed linear subspace and a pair of commuting operators $T_1,T_2$ on $\mathcal K$ such that
\begin{itemize}
\item[(i)] $(T_1+T_2,T_1T_2)$ is a $\Gamma$-contraction ;
\item[(ii)] $\mathcal H$ is a joint reducing subspace for $T_1+T_2$ , $T_1T_2$ ;
\item[(iii)] $(T_1+T_2)|_{\mathcal H}=S$ , $(T_1T_2)|_{\mathcal H}=P$ ;
\item[(iv)] $\max \{ \|T_1\|,\,\|T_2\| \} \leq 2$.
\end{itemize}
\end{thm}
Note that the above theorem was stated as Theorem \ref{thm:N03} and the next section is devoted to frame a proof of it.

\vspace{0.5cm}

\section{Proof of Theorem \ref{thm:N03}}\label{sec:04}

\vspace{0.5cm}

\noindent First we consider a proper subset of the symmetrized bidisc which will play an important role here. The \textit{symmetrized-half-bidisc} $\widehat{\mathbb G}$ is defined to be the symmetrization of the half bidisc $\dfrac{\mathbb D^2}{2}$ i.e.,
\[
\widehat{\mathbb G}=\left\{(w_1+w_2,w_1w_2)\,:\, |w_1|<\dfrac{1}{2}\,,\, |w_2|<\dfrac{1}{2}  \right\}.
\]
It is obvious that the \textit{closed symmetrized-half-bidisc} $\widehat{\Gamma}=\ov{\widehat{\mathbb G}}$ is the set
\[
\widehat{\Gamma}=\left\{(w_1+w_2,w_1w_2)\,:\, |w_1|\leq \dfrac{1}{2}\,,\, |w_2|\leq \dfrac{1}{2}  \right\}.
\]
\begin{defn}
A pair of commuting Hilbert space operators $(\widehat S, \widehat P)$ for which $\widehat \Gamma$ is a spectral set is called a $\widehat \Gamma$-\textit{contraction}.
\end{defn}
Needless to mention that every $\widehat \Gamma$-contraction is a $\Gamma$-contraction. Also it follows from the definition that for a $\widehat \Gamma$-contraction $(\widehat S, \widehat P)$, $\|\widehat S\|\leq 1$ and $\|\widehat P \|\leq \dfrac{1}{4}$. Note that a $\Gamma$-contraction $(S,P)$ for which $\|S\|\leq 1$ and $\|P\|\leq \dfrac{1}{4}$ may not be a $\widehat \Gamma$-contraction. For example, consider the point $(4/5,0)=\pi (4/5,0)$ belongs to $\mathbb G \subset \Gamma$ and thus a $\Gamma$-contraction and $\|S\|=|4/5|<1$ and $\|P\|=0<\dfrac{1}{4}$ here but $(4/5,0)$ does not belong to $\widehat \Gamma$. Below we have some features of the set $\widehat \Gamma$ and its interplay with $\Gamma$.

\begin{lem}\label{lem:N01}
$\widehat \Gamma$ is polynomially convex.
\end{lem}
Proof to this lemma is similar to that of Lemma 2.1 in \cite{ay-jfa}, where it was proved that the set $\Gamma$ is polynomially convex.

\begin{lem}\label{lem:N01A}
If a compact subset $K$ of $\mathbb C^n$ is polynomially convex then $K$ is a spectral set for a tuple of commuting Hilbert space operators $(T_1,\dots, T_n)$ if and only if the von-Neumann inequality holds for all polynomials in $\mathbb C[z_1,\dots,z_n]$, i.e.,
\[
\| f(T_1,\dots, T_n)\|\leq \|f\|_{\infty, K}\,, \quad \text{ for all } f\in \mathbb C[z_1,\dots,z_n].
\]
\end{lem}
We refrain from proving this lemma here. Indeed, a proof to this lemma is a routine exercise and follows from the fact that $K$ is polynomially convex.

\begin{lem}\label{lem:N02}
A point $(\widehat s, \widehat p)\in \widehat \Gamma \;(\text{or} \in \widehat{\mathbb G})$ if and only if $(2\widehat s, 4\widehat p)\in \Gamma \;(\text{or} \in \mathbb G)$.
\end{lem}
\begin{proof}
We prove for $\widehat \Gamma$ and $\Gamma$. The proof for $\widehat{\mathbb G}$ and $\mathbb G$ is similar.
Let $(\widehat s, \widehat p)\in \widehat \Gamma$. Then $\widehat s=w_1+w_2$ and $p=w_1w_2$ for a pair of points $(w_1,w_2)\in \dfrac{1}{2}\overline{\mathbb D^2}$. Clearly $2 \widehat s =(2w_1)+(2w_2)$ and $4\widehat p=(2w_1)(2w_2)$, where $2w_1,2w_2\in \overline{\mathbb D}$. therefore, $(2\widehat s, 4\widehat p)\in\Gamma$.

Conversely, suppose $(2\widehat s,4\widehat p)\in\Gamma$. Then $(2\widehat s,4\widehat p)=\pi (z_1,z_2)$ for some $z_1,z_2\in\overline{\mathbb D}$. Clearly $ \dfrac{z_1}{2}, \dfrac{z_2}{2} \in\dfrac{1}{2}\overline{\mathbb D}$ and their symmetrization gives $(\widehat s, \widehat p)$. Hence $(\widehat s, \widehat p)\in \widehat \Gamma$.
\end{proof}

An operator theoretic analogue of this result follows straightway.

\begin{lem}\label{thm:N01}
A pair of commuting Hilbert space operators $(\widehat S,\widehat P)$ is a $\widehat \Gamma$-contraction if and only if $(2\widehat S, 4\widehat P)$ is a $\Gamma$-contraction.
\end{lem}

\begin{proof}
Since both $\Gamma$ and $\widehat{\Gamma}$ are polynomially convex, it suffices to establish the von-Neumann inequality only.
Let $(\widehat S, \widehat P)$ be a $\widehat \Gamma$-contraction and let $\sigma:\widehat \Gamma \rightarrow \Gamma $ be defined by $\sigma(\widehat s, \widehat p)=(2\widehat s, 4\widehat p)$. Then $\sigma$ is a biholomorphic map. Let $f\in \mathbb C[z_1,z_2]$ be arbitrary. Then
\begin{align*}
\|f(2\widehat S, 4\widehat P) \| =\| f \circ \sigma (\widehat S, \widehat P) \|
& \leq \sup_{(w_1,w_2)\in \widehat \Gamma} |f\circ \sigma (w_1,w_2)| \\
& = \sup_{(w_1,w_2)\in \widehat \Gamma}|f(2w_1,4w_2)| \\
& = \sup_{(z_1,z_2)\in\Gamma}|f(z_1,z_2)| \; [\text{since } \sigma \text{ is biholomorphic}].
\end{align*}
Hence $(2\widehat S, 4 \widehat P)$ is a $\Gamma$-contraction.
Conversely, if $(2\widehat S, 4\widehat P)$ is a $\Gamma$-contraction, then for any $g\in\mathbb C[z_1,z_2]$,
\[
\| g(\widehat S, \widehat P) \|  =\| g \circ \sigma^{-1} (2\widehat S, 4\widehat P) \|
 \leq \| g\circ \sigma^{-1} \|_{\infty, \Gamma} =\|g\|_{\infty, \widehat \Gamma}.
\]
Therefore, $(\widehat S, \widehat P)$ is a $\widehat \Gamma$-contraction and the proof is complete.
\end{proof}

\noindent The next theorem is the major step to finish the proof.

\begin{thm}\label{thm:N02}
Let $(\widehat S,\widehat P)$ acting on $\mathcal H$ be a $\widehat \Gamma$-contraction. Then there exist a Hilbert space $\mathcal K$ that contains $\mathcal H$ as a closed linear subspace, a pair of commuting contractions $A,B$ on $\mathcal K$ such that $\mathcal H$ is a joint reducing subspace for $A+B, AB$ and that $(A+B)|_{\mathcal H}=\widehat S$, $(AB)|_{\mathcal H}=\widehat P$.
\end{thm}

\begin{proof}
Let $(S,P)=(2\widehat S, 4\widehat P)$. Then $(S,P)$ is a $\Gamma$-contraction by Lemma \ref{thm:N01}. Let $\mathcal K=\mathcal H \oplus \mathcal H$ and consider $\widetilde A, \widetilde B$ on $\mathcal K$ defined by
\[
\widetilde A =
\begin{bmatrix}
S/2 & (S^2-4P)/4 \\
I & S/2
\end{bmatrix}
\text{ and }
\widetilde B =
\begin{bmatrix}
S/2 &-(S^2-4P)/4 \\
-I & S/2
\end{bmatrix}
\]

Then
\[
\widetilde A +\widetilde B=
\begin{bmatrix}
S&0 \\
0&S
\end{bmatrix}
\; \text{ and } \; \widetilde A \widetilde B=\widetilde B \widetilde A =
\begin{bmatrix}
P&0 \\
0&P
\end{bmatrix}.
\]
Evidently $(\widetilde A + \widetilde B, \widetilde A \widetilde B)$ is a $\Gamma$-contraction by being the direct sum of two $\Gamma$-contractions and it follows from part-(4) of Theorem \ref{char}. Then by Lemma \ref{thm:N01}, $\left(\dfrac{1}{2}(\widetilde A+\widetilde B),\dfrac{1}{4}\widetilde A \widetilde B \right)$ is a $\widehat \Gamma$-contraction and
\begin{equation}\label{eqn:N01}
\left(\dfrac{1}{2}(\widetilde A+\widetilde B),\dfrac{1}{4}\widetilde A \widetilde B \right)=
\left(
\begin{bmatrix}
\widehat S &0 \\
0&\widehat S
\end{bmatrix}\,,\,
\begin{bmatrix}
\widehat P &0 \\
0 & \widehat P
\end{bmatrix}
\right)
=\pi \left( \dfrac{\widetilde A}{2}\,, \dfrac{\widetilde B}{2} \right).
\end{equation}
Setting $A=\dfrac{\widetilde A}{2}$ and $B=\dfrac{\widetilde B}{2}$, it follows from (\ref{eqn:N01}) that
\[
\left(
\begin{bmatrix}
\widehat S &0 \\
0&\widehat S
\end{bmatrix}\,,\,
\begin{bmatrix}
\widehat P &0 \\
0 & \widehat P
\end{bmatrix}
\right)
=(A+B,AB).
\]
Thus $\mathcal K$ is a joint reducing subspace for $A+B,AB$ and that $(A+B)|_{\mathcal H}=\widehat S$, $(AB)|_{\mathcal H}=\widehat P$. It remains to show that $\|A\|,\|B\|$ are not greater than $1$. Since $(S,P)$ is a $\Gamma$-contraction, by Theorem \ref{char}, $\Gamma$ is a complete spectral set for $(S,P)$. Considering the matricial polynomial
$F(s,p)=
\begin{bmatrix}
s/2 & (s^2-4p)/4 \\
1 & s/2
\end{bmatrix}
$
we have that
\[
\| \widetilde A \|= \left\| 
\begin{bmatrix}
S/2 & (S^2-4P)/4 \\
I & S/2
\end{bmatrix} 
\right\| \leq \sup_{(s,p)\in \Gamma} \|F(s,p)\|.
\]
For any $(s,p)=(z_1+z_2,z_1z_2)\in\Gamma$, we have that $|s|\leq 2$ and $|(s^2-4p)/4|=|(z_1-z_2)^2/4|\leq 1$. Since for any matrix $Q=[q_{ij}]_{n\times n}$ , $\|Q \|\leq n.\, {\displaystyle \max_{i,j} {|q_{ij}|} }$, it follows that $\| \widetilde A \| \leq 2$. Similarly $\|\widetilde B \|\leq 2$. It follows that $A,B$ are contractions and the proof is complete.

\end{proof}

\noindent The rest of the proof of Theorem \ref{thm:N03} follows as a corollary of Theorem \ref{thm:N02}. Consider the $\widehat \Gamma$-contraction $(\widehat S, \widehat P)=(S/2, P/4)$. If $(\widehat S, \widehat P)$ is the restriction of $(A+B,AB)$ to $\mathcal H$ for a pair of commuting contractions $A,B$ acting on $\mathcal H \oplus \mathcal H$, then $(S,P)$ is the restriction of $(T_1+T_2,T_1T_2)$ to the joint reducing subspace $\mathcal H$, where $T_1=2A\,,\,T_2=2B$. Evidently $\|T_1\|, \|T_2\| \leq 2$. Also, $(T_1+T_2,T_1T_2)=\left(
\begin{bmatrix}
S&0 \\
0&S
\end{bmatrix},
\begin{bmatrix}
P&0 \\
0&P
\end{bmatrix}
\right)$
on $\mathcal H\oplus \mathcal H$, which is a $\Gamma$-contraction by being the direct sum of two $\Gamma$-contractions. This completes the proof. \qed

\vspace{0.5cm}

\section{The distinguished boundary $b\Gamma$ and the $\Gamma$-unitaries}\label{sec:05}

\vspace{0.5cm}

\noindent We begin this Section with a new description for the points in the distinguished boundary $b\Gamma$ of $\Gamma$.

\begin{lem}\label{lem:3-01}
Let $(s,p)$ be a point in $\mathbb C^2$. Then $(s,p)\in b\Gamma$ if and only if $s=\overline{s}p$ and $|s^2-4p|+|s|^2=4$.
\end{lem}

\begin{proof}
Let $(s,p)\in b\Gamma$. Then $s=z_1+z_2\,, \, p=z_1z_2$ for some $z_1,z_2\in\mathbb T$. Therefore, $s=z_1+z_2=(\overline{z_1+z_2})z_1z_2=\overline{s}p$. Also
\[
|s^2-4p|+|s|^2=|z_1-z_2|^2+|z_1+z_2|^2=2(|z_1|^2+|z_2|^2)=4.
\]

Conversely, suppose $s=\overline{s}p$ and $|s^2-4p|+|s|^2=4$. Therefore, we have that
\begin{equation}\label{eqn:3-01}
2|s-\overline{s}p|+|s^2-4p|+|s|^2=4.
\end{equation}
So by part-(iii) of Theorem \ref{thm:char-gamma}, $(s,p)\in \Gamma$. If $s\neq 0$, then $s=\overline{s}p$ implies that $|p|=1$. If $s=0$, then from (\ref{eqn:3-01}) we have $|p|=1$. Thus $(s,p)\in \Gamma$ and $|p|=1$. So, by Theorem \ref{thm:char-DB} $(s,p)\in b\Gamma$ and the proof is complete.

\end{proof}

 In the literature, we have several characterizations for the points in $\Gamma$ and $\mathbb G$ (e.g., see \cite{ay-jot, tirtha-sourav}). Note that $\mathbb G$ is properly contained in $\Gamma \setminus b\Gamma$. The following lemma provides a characterization for the points in $\Gamma \setminus b\Gamma$.

\begin{lem}
Let $(s,p)\in \mathbb C^2$. Then $(s,p)\in \Gamma \setminus b\Gamma$ if and only if $|p|\neq 1$ and $|s-\overline{s}p|+|p|^2\leq 1$.
\end{lem}

\begin{proof}

Let $(s,p)\in\Gamma \setminus b\Gamma$. Then by Theorem \ref{thm:char-DB}, $|p|<1$ and by part-(ii) of Theorem \ref{thm:char-gamma} we have that $|s-\overline{s}p|+|p|^2 \leq 1$.

Conversely, suppose $|p|\neq 1$ and $|s-\overline{s}p|+|p|^2\leq 1$. Then $|s-\overline{s}p|\leq 1-|p|^2$ and hence $|p|\leq 1$. Since $|p|\neq 1$, we have that $|p|<1$. Clearly $\dfrac{s-\overline{s}p}{1-|p|^2}\leq 1$ and choosing $\beta = \dfrac{s-\overline{s}p}{1-|p|^2}$, we see that $|\beta|\leq 1$ and $s=\beta + \overline{\beta}p$. Thus by part-(iv) of Theorem \ref{thm:char-gamma}, $(s,p)\in\Gamma$. Again since $|p|<1$, by Theorem \ref{thm:char-DB} $(s,p)\in \Gamma \setminus b\Gamma$ and the proof is complete.

\end{proof}

The following result is simple but useful.

\begin{lem}\label{lem:simple}
If $S,P$ are commuting operators such that $S=S^*P$ and $P$ is normal $($or hyponormal$)$, then $S$ is normal $($or hyponormal respectively$)$.
\end{lem}

\begin{proof}
The proof follows from the fact that
\begin{align*}
S^*S-SS^*=(P^*S)S-(S^*P)(P^*S) = P^*(S^*P)S-S^*PP^*S 
& = S^*(P^*P)S-S^*(PP^*)S \\& 
=S^*(P^*P-PP^*)S.
\end{align*}

\end{proof}

We now present a new characterization for the $\Gamma$-unitaries.
\begin{thm}\label{thm:Gamma-unitary}
Let $(S,P)$ be a pair of commuting Hilbert space operators. Then $(S,P)$ is a $\Gamma$-unitary if and only if $P$ is a normal operator, $S=S^*P$ and $\sqrt{(S^2-4P)^*(S^2-4P)} +S^*S =4I$.
\end{thm}

\begin{proof}

Let $(S,P)$ be a $\Gamma$-unitary. Then by definition $S,P$ are commuting normal operators. Also, by Theorem \ref{thm:char-gamma-unitary}, $S=S^*P$, $\|S\|\leq 2$ and $S=U_1+U_2,\, P=U_1U_2$ for a pair of commuting unitary operators $U_1,U_2$. Since $U_1,U_2,U_1^*,U_2^*$ are all commuting unitary operators, a simple calculation yields
\begin{align*}
& \sqrt{(S^2-4P)^*(S^2-4P)} +S^*S \\ & =\sqrt{(U_1^*-U_2^*)^2(U_1-U_2)^2}+(U_1+U_2)^*(U_1+U_2) \\& =4I.
\end{align*}
Conversely, suppose $P$ is normal operator satisfying $S=S^*P$ and
\begin{equation}\label{eqn:3-02}
\sqrt{(S^2-4P)^*(S^2-4P)}+S^*S=4I.
\end{equation}
Then by Lemma \ref{lem:simple}, $S$ is normal and
by the positivity of $4I-S^*S$, it follows that $\|S/2\| \leq 1$ which is same as saying that $\|S\|\leq 2$. Also from (\ref{eqn:3-02}), we have that
\[
(S^2-4P)^*(S^2-4P)=(4I-S^*S)^2\,,
\]
which on evaluation gives
\[
{S^*}^2S^2-4{S^*}^2P-4P^*S^2+16P^*P=16I+S^*SS^*S-8S^*S.
\]
Along with the fact that $S=S^*P$, this leads to $$S^*(S^*S-SS^*)S+16(P^*P-I)=0$$ from which it follows that $P^*P=I$ as $S$ is normal. Now the normality of $P$ concludes that $P$ is unitary. Hence by part-(ii) of Theorem \ref{thm:char-gamma-unitary}, $(S,P)$ is a $\Gamma$-unitary and the proof is complete.

\end{proof}

%\vspace{0.5cm}

\section{Answer to Question 2}\label{sec:06}

\vspace{0.4cm}

\noindent In \cite{sourav}, an explicit minimal $\Gamma$-unitary dilation was constructed for a $\Gamma$-contraction $(S,P)$ which is analogous to the Schaeffer's minimal unitary dilation of a contraction (see CH-I of \cite{Nagy}). The most surprising fact about this $\Gamma$-unitary dilation is that the minimal dilation space for $(S,P)$ is exactly equal to the minimal unitary dilation space of the contraction $P$.

\begin{thm}[\cite{sourav}, Theorem 4.3]\label{main-dilation-theorem}
 Let $(S,P)$ be a $\Gamma$-contraction defined on a Hilbert space $\mathcal H$. Let $F$ and $F_*$ be the
  fundamental operators of $(S,P)$ and its adjoint $(S^*,P^*)$ respectively. Let
 $\mathcal K_0=\cdots\oplus\mathcal D_P\oplus\mathcal D_P\oplus\mathcal D_P\oplus\mathcal H\oplus
 \mathcal D_{P^*}\oplus\mathcal D_{P^*}\oplus\mathcal D_{P^*}\oplus\cdots=l^2(\mathcal D_P)\oplus
 \mathcal H\oplus l^2(\mathcal D_{P^*})$. Consider the operator pair $(T_0,U_0)$ defined on $\mathcal K_0$ by
\begin{align*}
& T_0(\cdots,h_{-2},h_{-1},\underbrace{h_0},h_1,h_2,\cdots) \\&
=(\cdots,Fh_{-2}+F^*h_{-1},Fh_{-1}+F^*D_Ph_0-F^*P^*h_1,
\\& \quad \quad \underbrace{Sh_0+D_{P^*}
F_*h_1},F_*^*h_1+F_*h_2,F_*^*h_2+F_*h_3,\cdots)\\ &
U_0(\cdots,h_{-2},h_{-1},\underbrace{h_0},h_1,h_2,\cdots) \\&
= (\cdots,h_{-2},h_{-1},D_Ph_0-P^*h_1,\underbrace{Ph_0+D_{P^*}h_1},h_2,h_3\cdots),
\end{align*}
where the $0$-th position of a vector in $\mathcal K_0$ has been indicated by an under brace. Then
$(T_0,U_0)$ is a minimal $\Gamma$-unitary dilation of $(S,P)$.
\end{thm}

The matrices of $T_0$ and $U_0$ with respect to the orthogonal
decomposition $\cdots\oplus\mathcal D_P\oplus\mathcal
D_P\oplus\mathcal D_P\oplus\mathcal H\oplus \mathcal
D_{P^*}\oplus\mathcal D_{P^*}\oplus\mathcal D_{P^*}\oplus\cdots = l^2(\mathcal D_P)\oplus \mathcal H\oplus
l^2(\mathcal D_{P^*})$
of $\mathcal K_0$
are the following:

\begin{equation}\label{2.3}
T_0 =\left[
\begin{array}{ c c c c|c|c c c c}
\bm{\ddots}&\vdots &\vdots&\vdots   &\vdots  &\vdots& \vdots&\vdots&\vdots\\
\cdots&F&F^*&0  &0&  0&0&0&\cdots\\
\cdots&0&F&F^*  &0&  0&0&0&\cdots\\
%\hline
\cdots&0&0&F  &F^*D_P&  -F^*P^*&0&0&\cdots\\ \hline

\cdots&0&0&0   &S&   D_{P^*}F_*&0&0&\cdots\\ \hline

\cdots&0&0&0   &0&  F_*^*& F_*&0&\cdots\\
\cdots&0&0&0   &0&  0&F_*^*&F_*&\cdots\\
\cdots&0&0&0  &0&   0& 0&F_*^*&\cdots\\
\vdots&\vdots&\vdots&\vdots&\vdots&\vdots&\vdots&\vdots&\bm{\ddots}\\
\end{array} \right],
\end{equation}

\vspace{5mm}

\begin{equation}\label{2.4}
U_0 = \left[
\begin{array}{ c c c c|c|c c c c}
\bm{\ddots}&\vdots &\vdots&\vdots   &\vdots  &\vdots& \vdots&\vdots&\vdots\\
\cdots&0&I&0  &0&  0&0&0&\cdots\\
\cdots&0&0&I  &0&  0&0&0&\cdots\\
%\hline
\cdots&0&0&0  &D_P&  -P^*&0&0&\cdots\\ \hline

\cdots&0&0&0   &P&   D_{P^*}&0&0&\cdots\\ \hline

\cdots&0&0&0   &0&  0& I&0&\cdots\\
\cdots&0&0&0   &0&  0&0&I&\cdots\\
\cdots&0&0&0  &0&   0& 0&0&\cdots\\
\vdots&\vdots&\vdots&\vdots&\vdots&\vdots&\vdots&\vdots&\bm{\ddots}\\
\end{array} \right].
\end{equation}

This minimal $\Gamma$-unitary dilation on the space $\mathcal K_0$ is unique upto a unitary. Indeed, if $(T,U_0)$ is another $\Gamma$-unitary dilation of $(S,P)$ on $\mathcal K_0$ then one can verify that there exists a unitary on $\mathcal K_0$ that intertwins $T$ and $T_0$. Now if $(S_1,P)$ is another $\Gamma$-contraction with fundamental operators $F_1, F_{1*}$ of $(S_1,P)$ and $(S_1^*,P^*)$ respectively, then by the construction of Theorem \ref{main-dilation-theorem} above, $(S_1,P)$ also possesses a minimal $\Gamma$-unitary dilation on the same dilation space $\mathcal K_0$. Thus if $(T_1,U_0)$ on $\mathcal K_0$ is the proposed $\Gamma$-unitary dilation of $(S_1,P)$, then $T_1$ has similar block-matrix representation like that of $T_0$ as in (\ref{2.3}) with $S,F, F_*$ being replaced by $S_1,F_1$ and $F_{1*}$ respectively in the expression. So, the answer to \textit{Question 2} of Section 1 is YES if and only if $[F,F_1]=0=[F^*,F]-[F_1^*,F_1]$ and $[F_*,F_{1*}]=0=[F_*^*,F_*]-[F_{1*}^*,F_{1*}]$. Therefore, \textit{Question 2} boils down to whether the fundamental operators $F, F_1$ of two $\Gamma$-contractions $(S,P)$ and $(S_1,P)$ satisfy $[F,F_1]=0=[F^*,F]-[F_1^*,F_1]$ when $S, S_1$ commute. Here we show by examples that none of them is true in general.

\begin{eg}
Let $\mathcal H$ be a Hilbert space and let $S,S_1,P$ be operators
on $\mathcal H \oplus \mathcal H$ given by
\[
S=\begin{bmatrix} Q&0\\ 0&0 \end{bmatrix}\,,\,
S_1=\begin{bmatrix} R&0\\ Y&R \end{bmatrix} \textup{ and } P=
\begin{bmatrix} W&0\\ 0&0 \end{bmatrix}\,,
\]
where $Q,R,W,Y$ are operators on $\mathcal H$ that satisfy the following:
\begin{itemize}
\item[(i)] $QW=WQ$,
$RW=WR$ and $YW=YQ=0$ so that $S,S_1,P$ commute;
\item[(ii)] $Y^*W \neq 0$ ;
\item[(iii)] $\|S \|<2,\| S_1\|<2$ and $\|P\|<1$;
\item[(iv)] $D_W^{-1}(Q-Q^*W)D_W^{-1}, D_W^{-1}(R-R^*W)D_W^{-1}$ and $(Y-Y^*W)D_W^{-1}$ have norm less than $1$. 
\end{itemize}
Evidently $D_W$ is invertible because $\|P\|<1$.
To show the existence of such $Q,R,W,Y$, one can choose $\mathcal H$ to
be equal to $\mathbb C^2$ and
\[
 Q=\begin{pmatrix} q&0\\ 0&0 \end{pmatrix}\,,\,W=
\begin{pmatrix} 0&w\\ 0&0 \end{pmatrix} \textup{ and }
Y=\begin{pmatrix} y&0\\ 0&0 \end{pmatrix}.
\]
Also $R$ can be any
matrix with proper norm such that $RW=WR$. Now
\[
 S-S^*P=\begin{bmatrix} Q-Q^*W&0\\ 0&0 \end{bmatrix} \textup{ and }
S_1-S_1^*P=\begin{bmatrix} R-R^*W&0\\ Y-Y^*W&0 \end{bmatrix}.
\]
Also
\[
D_P^2=I-P^*P=\begin{bmatrix} D_W^2&0\\ 0&I \end{bmatrix}.
\]
Now if we set
\[
 F=\begin{bmatrix} D_W^{-1}(Q-Q^*W)D_W^{-1}&0\\ 0&0 \end{bmatrix}\,,\,
F_1=\begin{bmatrix} D_W^{-1}(R-R^*W)D_W^{-1}&0\\
(Y-Y^*W)D_W^{-1}&0
\end{bmatrix}\,,
\]
we see that $S-S^*P=D_PFD_P$ and $S_1-S_1^*P=D_PF_1D_P$ with $\|F\|, \|F_1\|<1$ by condition-(iv) above. Thus by part-(4) of Theorem \ref{char}, both $(S,P)$ and $(S_1,P)$ are $\Gamma$-contractions with fundamental operators $F$ and $F_1$ respectively. Evidently we can choose $Q,W$ and $R$ in such a way that the two operators $D_W^{-1}(Q-Q^*W)D_W^{-1}, D_W^{-1}(R-R^*W)D_W^{-1}$ do not commute and consequently we have that $FF_1\neq F_1F$. 
 
\end{eg}

\begin{eg}
Let $A,B,T$ be doubly commuting operators on a Hilbert space
$\mathcal H$ such that $A$ is normal, $B$ is non-normal and
$\|A\|, \|B\|,\|T\|< 1$ so that $\mathcal D_T=\mathcal H$. We set $S,S_1,P$ on
$\mathcal H \oplus \mathcal H$ in the following way:
$$ S=\begin{bmatrix} A&A^*\\ A^*T&A \end{bmatrix} \,,\, S_1
=\begin{bmatrix} B&B^*\\ B^*T&B \end{bmatrix}\,,\,
P=\begin{bmatrix} 0&I\\ T&0 \end{bmatrix}. $$ We choose $A,B,P$ in
such a way that the norms of $S$ and $S_1$ are less than $2$. A
straight forward computation shows that $P$ commutes with $S$ and
$S_1$. Also
$$ SS_1=\begin{bmatrix} AB+A^*B^*T&AB^*+A^*B\\ A^*TB+AB^*T&A^*TB^*+AB \end{bmatrix} $$
which by symmetry in the entities is equal to $SS_1$ as $A,B,T$
doubly commute. Therefore, $S,S_1$ and $P$ commute. Now
$$ S-S^*P=\begin{bmatrix} A&A^*\\ A^*T&A \end{bmatrix}-
\begin{bmatrix} T^*AT&A^*\\ A^*T&A \end{bmatrix} =
\begin{bmatrix} (I-T^*T)A&0\\ 0&0 \end{bmatrix}
= \begin{bmatrix} D_T^2A&0\\ 0&0 \end{bmatrix}. $$ By a similar
computation we have that
$$ S_1-S_1^*P=\begin{bmatrix} D_T^2B&0\\ 0&0 \end{bmatrix}. $$
Also
$$ D_P^2=\begin{bmatrix} D_T^2&0\\ 0&0 \end{bmatrix}. $$
Therefore, $\mathcal D_P \subseteq \mathcal H \oplus \{0 \}$ and
if we set $F$ and $F_1$ as
$$ F= \begin{bmatrix} A&0\\ 0&0 \end{bmatrix} \,,\,
F_1=\begin{bmatrix} B&0\\ 0&0 \end{bmatrix} $$ then it is evident
that $\|F\|, \|F_1\|$ are less than $1$ and that
$$ S-S^*P=D_PFD_P \,,\, S_1-S_1^*P=D_PF_1D_P. $$
Therefore, by part-(4) of Theorem \ref{char}, $(S,P)$ and
$(S_1,P)$ are $\Gamma$-contractions with fundamental operators $F,F_1$ respectively. Since $A$ is normal and $B$
is non-normal, it follows that $ 0=F^*F-FF^* \neq F_1^*F_1-F_1F_1^*$.

\end{eg}

\end{document}